\documentclass{amsart}
\usepackage{amsmath}
\usepackage{amssymb}
\usepackage{amsthm}
\usepackage{amsfonts}
\usepackage{mathrsfs}
\usepackage{amscd}
\usepackage[hyphens]{url}

\newtheorem{theorem}{Theorem}[section]
\newtheorem{proposition}[theorem]{Proposition}
\newtheorem{lemma}[theorem]{Lemma}
\newtheorem{corollary}[theorem]{Corollary}
\newtheorem{assumption}[theorem]{Assumption}
\newtheorem{definition}[theorem]{Definition}
\theoremstyle{definition}
\newtheorem{example}[theorem]{Example}
\newtheorem{algorithm}[theorem]{Algorithm}
\newtheorem{remark}[theorem]{Remark}

\newcommand{\C}{\ensuremath{\mathbb C}}

\newcommand{\R}{\ensuremath{\mathbb R}}

\newcommand{\N}{\ensuremath{\mathbb N}}
\newcommand{\T}{\ensuremath{\mathbb T}}
\newcommand{\K}{\ensuremath{\mathbb K}}
\newcommand{\cls}{{\text{\upshape cls}}}

\newcommand{\rank}{{\text{\upshape rank}}}
\newcommand{\crk}{{\text{\upshape corank}}}
\newcommand{\vect}{{\text{\upshape vec}}}
\newcommand{\ptdeg}{\preceq_{\text{\upshape tdeg}}}
\newcommand{\lt}{{\text{\upshape lt}_{\prec}}}
\newcommand{\lee}{{\text{\upshape le}}}

\newcommand{\A}{\mathcal{A}}
\newcommand{\Kk}{\mathcal {K}}

\newcommand{\Aa}{\mathcal{A}}
\newcommand{\Bb}{\mathcal{B}}
\newcommand{\Cc}{\mathcal{C}}

\newcommand{\tdeg}{\prec_{\text{\upshape tdeg}}}

\def\rank{\mbox{rank}}

\begin{document}

\title {A Certificate for Semidefinite
Relaxations in Computing Positive Dimensional Real Varieties}


\author{
        Yue  Ma, Chu Wang  and Lihong Zhi}

\footnote{KLMM, Academy of Mathematics and Systems Science, CAS,
Beijing, 100190, China, \\
 \{yma, cwang, lzhi\}@mmrc.iss.ac.cn}

\begin{abstract}
For an ideal I with a positive dimensional real variety, based on
moment relaxations, we study how to compute a  Pommaret basis which
is simultaneously  a Groebner basis of an ideal J generated by  the
kernel of a truncated moment matrix and nesting between I and its
real radical ideal. We provide a certificate consisting of a
condition on coranks of moment matrices for terminating
 the algorithm.  For a generic delta-regular coordinate system,
we prove that the condition is satisfiable in
 a  large enough   order of
moment relaxations.

\end{abstract}


\maketitle

\section{introduction}
Finding real solutions of a polynomial system is a classical
mathematical problem with wide applications. Let $I=\langle h_1,\ldots, h_m \rangle\subseteq\R[x]:=\R[x_1,\ldots,x_n]$ be an ideal
generated by polynomials $h_1,\dots,h_m\in \R[x]$. Its complex and
real algebraic varieties are defined as
\[
V_{\C}(I):=\{x\in\C^n\mid f(x)=0 ~\forall\,f\in I\}, \quad
V_{\R}(I):= V_{\C}(I) \cap \R^n.
\]
The vanishing ideal of a set $V\subseteq\C^n$ is an ideal
\[
I(V):=\{f\in\C[x]\mid f(v)=0, ~\forall~v\in V\}.
\]
The radical (also called complex radical) of $I$ is
\[
\sqrt{I}:=\left\{f\in\C[x] \left| f^k \in I  \right. ~\text{for some
} k \in \N  \right\},
\]
while the real radical of $I$ is defined as
\[
\sqrt[\R]{I}:=\left\{f\in\R[x] \left| f^{2k}+\sum_{i=1}^r q_i^2\in I
\right. ~\text{for some } k \in \N, \, q_1,\ldots, q_r\in\R[x]
\right\}.
\]
Clearly, they satisfy the inclusion $I \subseteq \sqrt{I} \subseteq
\sqrt[\R]{I}$. An ideal $I$ is called \textit{radical} (resp.
\textit{real radical}) if $I= \sqrt{I}$ (resp. $I=\sqrt[\R]{I}$).
According to the Real Nullstellensatz \cite{BCR87}, the vanishing
ideal $I(V_{\R}(I))$ of the zero set $V_{\R}(I)$ is a real radical
ideal and $I(V_{\R}(I))= \sqrt[\R]{I}$.

There exists much work on computing a complex radical ideal
$\sqrt{I}$, like \cite{BW96,EHV92,GTZ88,KL91,Lasserre09}. The
algorithms range from numerical ones (e.g.,
\cite{JMRS2012,LLR09a,LLR09b}) to symbolic ones  (e.g.,
\cite{BW93,Seidenberg74}). 
For the general case of $I$ being
positive dimensional, a commonly used technique is to reduce the
problem to the zero-dimensional case, like in  Gianni et al.
\cite{GTZ88} and Krick and Logar \cite{KL91}.

The problem of computing the real radical ideal $\sqrt[\R]{I}$ is
typically much more difficult than computing $\sqrt{I}$.  Becker and
Neuhaus \cite{BN93} proposed a symbolic algorithm based on the
primary decomposition to compute $\sqrt[\R]{I}$ (also see
\cite{Neuhaus98,Spang:2007,Xia02,zeng99}).
Some interesting 
algorithms based on critical point methods were proposed in
\cite{Aubry02,BGHM01,Basu97,Safey03} to compute a point on each
semi-algebraically connected component of real algebraic varieties.

A new  approach based on moment relaxations has been proposed by
Lasserre et al. \cite{LLMTR11,LLR09a,LLR09b,LR10}  for computing
$\sqrt[\R]{I}$,  provided it is a zero-dimensional variety.  Hereby
we briefly describe this interesting approach.

For a sequence $y=(y_\alpha)_{\alpha\in\N^n}\in\R^{\N^n}$, its
\textit{moment matrix}
\[M(y):=(y_{\alpha+\beta})_{\alpha,\beta\in\N^n}\]
is a real symmetric matrix whose rows and columns are indexed by the
 set $\T^n:=\{x^{\alpha} \mid \alpha \in \N^n\}$ of
monomials.  Given a polynomial $h\in \R[x]$, set
$\vect(h):=(h_{\alpha})_{\alpha\in\N^n}$ and define the sequence
$hy:=M(y)\vect(h)\in \R^{\N^n}$. We say that a polynomial $p$ lies
in the kernel of $M(y)$ when $M(y)p:=M(y)\vect(p)=0$. Given a
truncated moment sequence
$y=(y_\alpha)_{\alpha\in\N^n_{2t}}\in\R^{\N^n_{2t}}$, it defines a
\textit{ truncated moment matrix}
\[M_t(y):=(y_{\alpha+\beta})_{\alpha,\beta\in\N^n_t}\]
indexed by the set
$\T_{ t}^n:=\{x^{\alpha} \mid \alpha \in\N^n_t ~{\text{with}}~
|\alpha|:=\Sigma_{i=1}^n \alpha_i \leq t\}$.

We work with the space $\R[x]_t$ of polynomials of the degree
smaller than or equal to $t$. For a polynomial $p\in \R[x]_t$, if
$M_t(y)\vect(p)=0$, we
 say $ p$ lies in the kernel of $M_t(y)$, i.e.,
\begin{equation}\label{kernal}
    \ker M_t(y):=\{p\in\R[x]_{t}\mid M_t(y)\vect(p)=0 \}.
\end{equation}

Let $I=\langle h_1, \ldots,h_m\rangle \subseteq \R[x]$ be an ideal
and set
\begin{equation}\label{e,4}
d_j:=\lceil\deg(h_j)/2\rceil, \quad d:=\max_{1\leq j\leq m}d_j.
\end{equation}
For $t\geq d$, define the set
\begin{equation}\label{dual space}
    \Kk_t:=\{y\in\R^{\N^n_{2t}}\mid y_0=1, M_t(y)\succeq0, M_{t-d_j}(h_jy)=0,
    j=1,\ldots,m\}.
\end{equation}
An element $y\in\Kk_t$ is \textit{generic} if $M_t(y)$ has maximum
rank over $\Kk_t$. We denote
\begin{equation}\label{genericelement}
\Kk_t^{gen}:=\{ y \in \Kk_t \mid ~\text{rank} \, M_t(y) ~{\rm is ~maximum~
over}~ \Kk_t\}.
\end{equation}

When the real algebraic variety $V_{\R}(I)$ is finite, Lasserre et
al.\cite{LLR08} used the flat extension (a rank condition of moment
matrices in \cite{CF96}) as a certificate to check whether
polynomials in $\ker M_s(y)\,(1\leq s \leq t)$ for a generic element
$y\in \Kk_t$ generates the real radical ideal $I(V_{\R}(I))$. When
$V_{\R}(I)$ is positive dimensional, this certificate does not work.
 The example given by Fialkow in \cite[Example
3.2]{Fialkow11} can be used to explain the difficulty. Unlike the
zero-dimensional case, although the kernel of the moment matrix of
the third order  consists  of only a polynomial $z-x^3$ which is
already  a Gr\"obner basis  of the real radical 
ideal $I=I(V_\R(I))=\langle z-x^3\rangle$, we can not extend the truncated moment
sequence $y \in  \Kk_3$ to the next order, i.e.,
 $y$  has no representing measure.


 The
motivation of this paper is to provide a certificate for checking
$\langle \ker
M_t(y) 
\rangle =I(V_{\R}(I))$ when $V_{\R}(I)$ is positive dimensional.
Unfortunately, we still can not solve this open problem  \cite[\S
2.4.3]{LR10} completely. However, we provide a certificate
(\ref{e,5}) based on the geometric involutivity theory
\cite{Robin:2006,SRWZ09,Seiler2010} 
  for checking  whether we have
obtained a weak 
 Pommaret basis (also a Gr\"obner basis) of an ideal
$J=\langle \ker M_{t-2}(y) \rangle$
%
satisfying $ I\subseteq J \subseteq I(V_{\R}(I))$ 
under
graded reverse lexicographic order.  A (weak) Pommaret basis is  a
special form of the familiar Gr\"obner basis which allows for
directly reading off the depth, the projective dimension and the
Castelnuovo-Mumford regularity of a module. When the real algebraic
variety $V_{\R}(I)$ is positive dimensional, for examples in Section
\ref{examples}, we succeed in showing that the computed Pommaret
basis is an involutive basis of the real radical ideal
$I(V_{\R}(I))$.  In general, it is still not possible to prove that
the kernel of the moment matrix satisfying the certificate
(\ref{e,5}) generates a real radical ideal.

\bigskip
The paper is organized as follows. In Section \ref{preliminary}, we
review some preliminary backgrounds like elementary algebraic
geometry, moment matrices, involutive divisions and involutive
bases. In Section \ref{mainresults}, we  present an algorithm based
on the  semidefinite programming and  moment relaxations in
computing a Pommaret basis of an ideal $J$ 
satisfying $I \subseteq J \subseteq I(V_{\R}(I))$, $V_\R(I)=V_\C(J)
\cap \R^n$, and  propose a certificate for terminating the algorithm
and prove it works for a positive dimensional $V_{\R}(I)$ under a
$\delta$-regular coordinate system. In Section \ref{examples}, we
present computational results for a set of examples in
\cite{Rostalski09,SRWZ09, Seiler02,Stetter04}. 
 Some open
questions and ongoing work are given in Section \ref{conclusions}.

\section{preliminary}\label{preliminary}

We introduce 
some notations and preliminaries about polynomials, matrices,
semidefinite programs and the involution. Given $\K=\R$ or $\C$, the
ring of multivariate polynomials in $n$ variables over the field
$\K$ is denoted by $\K[x]:=\K[x_1,\ldots,x_n]$. For an integer
$t\geq0$, $\K[x]_{t}$ denotes the set of polynomials of degree at
most $t$. $\N$ denotes the set of nonnegative integers and we set
$\N_{t}^n:=\{\alpha\in\N^n\mid |\alpha|:=\Sigma_{i=1}^n \alpha_i
\leq t\}$ for $t\in\N$. For $\alpha\in \N^n$, $x^{\alpha}$ denotes
the monomial $x_1^{\alpha_1}\cdots x_n^{\alpha_n}$ whose total
degree is $|\alpha|:=\Sigma_{i=1}^n \alpha_i$. All monomials are
included in $\T^n:=\{x^{\alpha} \mid \alpha \in \N^n\}$ and $\T_{
t}^n:=\{x^{\alpha} \mid \alpha\in\N_{t}^n\}$ consists of monomials
with degrees bounded by $t\in\N$. Consider a polynomial $p\in\K[x]$,
$p=\Sigma_{\alpha\in\N^n}p_{\alpha}x^{\alpha}$, where there are only
finitely many nonzero $p_{\alpha}\in\K$, its leading term $\lt(p)$
is the maximum term $x^{\alpha}$ with respect to a monomial order
$\prec$ for which $p_{\alpha}\neq 0$. We denote by $\langle
\lt(I)\rangle$ the ideal generated by leading
terms of polynomials in $I$. 
The symbol $[x]_t$ denotes the sequence consisting 
 of all monomials of degrees less than or equal to $ t$:
 \[
[x]_t := [1,  x_1, \cdots, x_n, x_1^2, x_1 x_2, \cdots, x_1^t,
x_1^{t-1}x_2, \cdots, x_n^t].
\]


\subsection{Properties of moment matrix}
The kernel of a moment matrix is particularly useful as it has the
following properties, see
\cite{CF96,LLR08,Laurent05,Laurent09,Moller04}.

\begin{lemma} \cite[Proposition 3.6]{LLR08}
Let $\ker M(y):=\{p\in\R[x]\mid M(y)\vect(p)=0\}$
 be the kernel of a moment matrix $M(y)$. Then
  $\ker M(y)$
 is an ideal in $\R[x]$. Moreover, if $M(y)\succeq0$, then $\ker
M(y)$ is a real radical ideal.
\end{lemma}

The kernel of the truncated moment matrix $M_t(y)$  is not an ideal,
but under certain conditions, it has the following properties.

\begin{proposition}\cite[Lemma 3.5, 3.9]{LLR08}\label{t1}
Let $y\in\R^{\N^n_{2t}}$  and its truncated  moment matrix $M_t(y)$
is positive semidefinite.
\begin{itemize}
\item [(i)] If $f, g\in\R[x]$ with $\deg(fg)\leq t-1$, then
$f\in\ker M_t(y) \Longrightarrow fg\in\ker M_t(y)$.

\item [(ii)] For a polynomial $p\in \R[x]$, if
$p^{2k}+ \sigma \in \ker M_t(y)$ for some $k\in\N$ and $\sigma \in \sum
\R[x]^2$, then $p\in\ker M_t(y)$.

\item [(iii)] We have   $\ker M_t(y)\cap\R[x]_{s} =
    \ker M_s(y) \quad\text{for} ~1\leq s\leq t$.
\end{itemize}
\end{proposition}


Generic elements of $\Kk_t$ have useful properties.
 The following results are cited from
\cite[Lemma 3.1]{LLR08} and \cite[Lemma 7.28, 7.39]{Rostalski09}.
\begin{proposition}\label{t2}
Assume $y\in \Kk_t^{gen}$ is generic.
\begin{itemize}
\item [(i)]
For all $1\leq s \leq t$, we have $\ker M_s(y)\subseteq \sqrt[\R]{I}$ and
$\ker M_s(y)\subseteq \ker M_s(z)$ for all $z\in\Kk_t$.

\item [(ii)]
If $t\leq t'$ and $y'\in \Kk_{t'}^{gen}$, then $ \ker
M_t(y) \subseteq \ker  M_{t'}(y').$

\item [(iii)]
For every finite basis $\{g_1,\ldots,g_k\}$ of the real radical ideal
$\sqrt[\R]{I}$, there exists $t_0\in\N$ such that $g_1,\ldots,g_k\in
\ker M_t(z)$ for all $z\in\Kk_t$ and $t\geq t_0$.

\item [(iv)] It holds that
$\langle \ker M_t(y)\rangle=\sqrt[\R]{I}$ if $t$ is sufficiently
large.
\end{itemize}
\end{proposition}

In the following, we review some properties of moment matrices in the occurrence of inequality constraints. Consider the semialgebraic
set \begin{equation}
\label{e,1} \Aa:=\{x\in\R^n\mid f_1(x)\geq0, \ldots,
f_s(x)\geq0\},
\end{equation}
where $f_1,\ldots,f_s\in\R[x]$. The $\Aa$-variety $V_{\Aa}(I)$
denotes the intersection
\[V_{\A}(I)=V_{\R}(I)\cap \Aa.\]
For every $\nu \in \{0,1\}^s$, we denote the product $f^\nu :=
f_1^{\nu_1}f_2^{\nu_2}\cdots f_s^{\nu_s}$.

\begin{definition}\cite{Marshall08}\label{t3}
The $\mathcal{A}$-radical of an ideal $I$ is defined as
\[
 \sqrt[\Aa]{I}:=\left\{p\in\R[x]
\left| p^{2k}+\sum_{ \nu \in\{0,1\}^s}\sigma_\nu f^\nu \in I \right.
 \,~\text{for some}~  k\in\N, \, \sigma_\nu \in \sum \R[x]^2 \right\}.
 \]
The ideal $I$ is called $\Aa$-radical if $I=\sqrt[\Aa]{I}$.
\end{definition}


\begin{theorem}\label{t4}
\cite[Semialgebraic Nullstellensatz]{Stengle74} Let $I$ 
 be an ideal in $\R[x]$ and $\Aa$ be defined by  (\ref{e,1}). Then
$\sqrt[\mathcal{A}]{I}$ is an $\Aa$-radical ideal  and
$\sqrt[\Aa]{I}=I(V_{\R}(I)\cap\Aa).$
\end{theorem}

To compute the $\Aa$-radical ideal $\sqrt[\Aa]{I}$, we consider the set
\begin{equation}
\Kk_{t,\Aa}:= \Kk_t \cap \left\{y\in \R^{\N^n_{2t}}:
M_{t-d_{f^{\nu}}}(f^{\nu} y) \succeq 0, ~\forall \nu \in \{0,1\}^s
\right \},
\end{equation}
where $d_{f^{\nu}}=\lceil\deg(f^{\nu})/2\rceil$.
Clearly, the set $\Kk_{t,\Aa}$ is a restriction of $\Kk_t$.
The definition of the set $\Kk_{t,\A}$ is motivated by the
polynomials in $\sqrt[\A]{I}$ and the Semialgebraic Nullstellensatz.
The generic elements of $\Kk_{t,\Aa}$ are similarly defined to be
the elements of the set
\[
\Kk_{t,\Aa}^{gen}:=\{ y \in \Kk_{t,\Aa}: \rank \, M_t(y) \mbox{ is
maximum over }  \Kk_{t,\Aa}\}.
\]

\begin{lemma}\label{t5}
Let $\{g_1,\ldots,g_k\}$ be a set of generators for the ideal $\sqrt[\Aa]{I}$.
Then there exists $t_0\in\N$ such that $g_1,\ldots,g_k\in \ker M_t(y)$ for all $y\in\Kk_{t,\Aa}$ and $t\geq t_0$.
\end{lemma}

The following proof mimics the proof of Claim 4.7 in \cite{LLR08}.
\begin{proof}
For each $\ell=1,\ldots,k$, by Theorem \ref{t4}, there exists
$m_l\in\N$ and polynomials $\sigma_\nu \in\Sigma\R[x]^2$ and
$u_j\in\R[x]$ for $1\leq j \leq m$ such that
\begin{equation}
\label{e,2} g_l^{2m_l}+\sum_{ \nu
\in\{0,1\}^s } \sigma_\nu f^\nu = \sum_{j=1}^m u_jh_j.
\end{equation}

For $t\geq t_0$, where \[t_0=1+\max(d, \deg(g_l^{2m_l}),
\deg(\sigma_\nu f^\nu ),\deg(u_jh_j)),\]
 since $\deg (u_jh_j)\leq
t-1$ and $h_j\in\ker M_t(y)$,  by Proposition \ref{t1} (i), we have
$u_jh_j\in\ker M_t(y)$,
i.e., $g_l^{2m_l}+\sum_{\nu\in\{0,1\}^s}\sigma_{\nu} f^{\nu}\in\ker
M_t(y)$.
Set $\sigma_\nu=\sum_{j} \sigma_{\nu, j}^2\in\Sigma\R[x]^2$,
then we have
\begin{align*}
    \vect(g_l^{m_l})^T M_t(y)\,\vect(g_l^{m_l})+ \sum_{\nu, j} \vect(\sigma_{\nu, j})^T
    M_{t-d_{f^{\nu}}}(f^{\nu} y) \,\vect
    (\sigma_{\nu, j})=0. 
\end{align*}
Since $M_t(y)\succeq 0$ and 
$M_{t-d_{f^{\nu}}}(f^{\nu} y) \succeq 0$,
every summand in the above expression  
must be zero, and thus
$g_l^{m_l}\in \ker M_t(y)$. If $m_l$ is even, $g_l^{m_l}\in \ker
M_t(y)$ implies $g_l^{m_l/2}\in \ker M_t(y)$. If $m_l$ is odd,
 since $\deg(g_l^{m_l+1}) \leq t-1$, we have
\[
g_l^{m_l}\in \ker M_t(y) \Rightarrow g_l^{m_l+1}\in \ker M_t(y)
\Rightarrow g_l^{(m_l+1)/2}\in \ker M_t(y).
\]
Repeat this process, we can  show that $ g_l\in\ker M_t(y).$
\end{proof}

\begin{theorem}\label{t6}
There exists $t_0\in\N$ such that $\langle \ker
M_t(y)\rangle=\sqrt[\Aa]{I}$ for all $y\in \Kk_{t,\Aa}^{gen}$
and $t\geq t_0$.
\end{theorem}

\begin{proof}
Let $\{g_1,\ldots,g_k\}$ be a set of generators for the ideal $\sqrt[\Aa]{I}$. According to Lemma \ref{t5}, we can choose $t_0\in\N$ such that $g_1,\ldots,g_k\in \ker M_t(y)$ for all $y\in\Kk_{t,\Aa}$ and $t\geq t_0$.
Let $y\in \Kk_{t,\Aa}^{gen}$ and choose an arbitrary point $v\in
V_{\Aa}(I)$. Then $[v]_{2t} \in \Kk_{t,\Aa}$ and
$z=(y+[v]_{2t})/2 \in \Kk_{t,\Aa}$. Clearly, it holds that
\[
\ker M_t((y + [v]_{2t})/2) = \ker M_t(y) \cap \ker M_t([v]_{2t}).
\]
The rank of $M_t(y)$ being maximum implies $\ker M_t((y +
[v]_{2t})/2) = \ker M_t(y)$ and $\ker M_t(y) \subseteq
\ker M_t([v]_{2t})$. For every $p \in \ker M_t(y)$, we must have $p \in
\ker M_t([v]_{2t})$ and $p(v)=0$. This means that $p$ vanishes on the set
$V_{\Aa}(I)$. By Theorem \ref{t4}, we get $p\in \sqrt[\Aa]{I}$ and
thus the inclusion $\langle\ker M_{t}(y)\rangle\subseteq
\sqrt[\Aa]{I}=I(V_{\Aa}(I))$ holds. 
Since $g_1,\ldots,g_k\in \ker M_t(y)$, we get 
$\langle\ker M_{t}(y)\rangle\supseteq \sqrt[\Aa]{I}=I(V_{\Aa}(I))$
and the proof is completed. 
\end{proof}

\subsection{Involutive Divisions and Involutive Bases}

When the real algebraic variety $V_{\R}(I)$ is finite, Lasserre et
al. \cite{LLMTR11, LLR08} proposed  new approaches based on moment
relaxations for computing Gr\"obner bases or border bases of the
real radical ideal $\sqrt[\R]{I}$. For the positive dimensional real
variety $V_{\R}(I)$, we can also compute its Gr\"obner bases.
 Stimulated by the work in \cite{LLR09a} and
\cite{RZ09,Robin:2006,SRWZ09}, 
we propose a new approach based on the  completion to involution to
compute a  Pommaret basis of an ideal nested between $I$ and
$\sqrt[\R]{I}$. A Pommaret basis is simultaneously a Gr\"obner
basis, but  contains extra information such as the
Castelnuovo-Mumford regularity. Moreover, we provide a new  stopping
criterion for the algorithm  which is based on the classical
Cartan's test for  involution from the theory of exterior
differential systems. We now  introduce some basic concepts from the
classical theory of involutive systems for polynomial systems. For
background, see \cite{Seiler02,Seiler2010}.

\begin{definition}
Let $\nu=[\nu_1,\ldots,\nu_n]\in \N^n$ be the  multi index of a
monomial $x^\nu$. If $k$ is the smallest value such that
$\nu_{k}\neq0$, then the class of $\nu$ or $x^\nu$ is $k$, written
by $\cls(\nu)=k$ or $\cls(x^\nu)=k$. The class of a polynomial $f$
which is denoted by $\cls(f)$ is $k$, if the class of its leading term $\cls(\lt(f))=k$. 
\end{definition}

We say that a term order  \textit{respects classes}, if for
monomials $x^\mu$ and $x^\nu$ of the same total degree,
$\cls{(\mu)}<\cls{(\nu)}$ implies $x^\mu \prec x^\nu$.  An important
example  of a class respecting ordering is the  \textit{graded
reverse lexicographic} order $\tdeg$.

\begin{definition}
With an ordering on the variables $x_1 \prec \cdots\prec x_n$, the
graded reverse lexicographic order  $\tdeg$ is defined by $
x^{\alpha}\tdeg x^{\beta}$,  if $ |\alpha|<|\beta|$, or $
|\alpha|=|\beta|$ and the first non-vanishing entry of the multi
index  $\alpha-\beta$ is positive.
\end{definition}

Throughout the paper, we use $\tdeg$ in assigning orders of
monomials, and sorting rows and columns of a moment matrix $M_t(y)$.
Let  $(\N^n,+)$ be an Abelian monoid  with the addition defined
componentwise. For any multi index $\nu \in \N^n$, we introduce its
cone $\Cc(\nu)=\nu+\N^n$, i.e., the set of all multi indices that
can be reached from $\nu$ by adding another multi index.
%
%

\begin{definition}\cite[Definition 3.1.1]{Seiler2010}\label{indef}
An involutive division $L$ is defined on the monoid $(\N^n,+)$, if
for any finite subset $\Bb \subset \N^n, \ a \ set \
N_{L,\Bb}(\nu)\subseteq \{1,\ldots,n\}$ of multiplicative indices,
and consequently a submonoid $L(\nu,\Bb)=\{\mu \in \N^n \mid \forall
j \not\in N_{L,\Bb}(\nu) : \mu_j=0\}$,
 is associated to every multi index $\nu \in \Bb$ such that the following two conditions on the involutive cones
 $\Cc_{L,\Bb}(\nu)=\nu+L(\nu,\Bb)\subseteq \N^n$ are satisfied.
\begin{itemize}
\item [(i)]  If there exist two elements $\mu,\nu \in \Bb$ with $\Cc_{L,\Bb}(\mu) \cap \Cc_{L,\Bb}(\nu) \neq \emptyset$,
either $\Cc_{L,\Bb}(\mu)\subseteq\Cc_{L,\Bb}(\nu)$ or $\Cc_{L,\Bb}(\nu)\subseteq\Cc_{L,\Bb}(\mu)$ holds.
\item [(ii)] If $\Bb'\subset \Bb$, then $N_{L,\Bb}(\nu)\subseteq N_{L,\Bb'}(\nu)$ for all $\nu \in \Bb'$.
\end{itemize}
An arbitrary multi index $\mu \in \N^n$ is involutively divisible by $\nu \in \Bb$, written $\nu \mid_{L,\Bb}\mu$,
if $\mu \in \Cc_{L,\Bb}(\nu)$. In this case $\nu$ is called an involutive divisor of $\mu$.
\end{definition}


\begin{definition}\label{t9}\cite[Example 3.1.7]{Seiler2010} The  Pommaret
division written by $P$ is defined by  a simple rule: if
$\cls(\nu)=k$, then we set $N_{L,\Bb}(\nu)=\{1,\ldots,k\}$.
\end{definition}

\begin{remark}
The Pommaret division is a globally defined division as the
assignment of the multiplicative indices to a multi
index $\nu \in
\Bb$ is independent of the set $\Bb$. The
Pommaret division is an involutive division by \cite[Lemma 3.1.8]{Seiler2010}.
\end{remark}
\begin{definition}\cite[Definition 3.1.9]{Seiler2010}\label{f,1}
The involutive span of a finite set $\Bb\subset \N^n$ is
\begin{equation}
\label{e,3} \langle \Bb \rangle_L =\bigcup_{\nu\in\Bb}
\Cc_{L,\Bb}(\nu).
\end{equation}\\
The set $\Bb$ is called weakly involutive for the division $L$
or a weak involutive basis of the monoid ideal $\langle\Bb\rangle$,
if $\langle \Bb \rangle_L$=$\langle \Bb \rangle$.
The set $\Bb$ is a strong involutive basis or for short an involutive basis,
if the union (\ref{e,3}) is disjoint, i.e.,
 the intersections of the involutive cones are empty.
\end{definition}

For a polynomial $f \in \K[x]$ and a term order $\prec$, we select
its leading term $\lt (f)=x^\mu$ with the leading exponent
$\lee_{\prec} (f)=\mu$.

\begin{definition}\cite[Definition 3.4.1]{Seiler2010}\label{f,2}
Let $I\subseteq\K[x]$ be an ideal. A finite set $\mathcal{H}\subset
I$ is a \textit{weak involutive basis} of $I$ for an involutive
division $L$ on $\N^n$, if $\lee_{\prec}(\mathcal{H})$ is a weak
involutive basis of the monoid ideal $\lee_{\prec}(I)$. The set
$\mathcal{H}$ is a \textit{strong involutive basis} of $I$, if
$\lee_{\prec}(\mathcal{H})$ is a strong involutive basis of
$\lee_{\prec}(I)$ and two distinct elements of $\mathcal{H}$
 never possess the same leading exponents.
\end{definition}

\begin{remark}
Definition \ref{f,1} and Definition \ref{f,2}
imply immediately that any weak involutive basis is a Gr\"{o}bner
basis.
\end{remark}


Not every ideal in $\K[x]$ possesses a finite Pommaret basis (see
\cite{Seiler2010}).

\begin{definition}\cite[Definition 4.3.1]{Seiler2010}
A coordinate system is called $\delta$-regular for the ideal
$I\subseteq  \K[x]$ and the term order $\prec$, if $I$ possesses a
finite Pommaret basis for the term order $\prec$. 
\end{definition}
\begin{theorem}\cite[Theorem 4.3.15]{Seiler2010}
Every polynomial ideal $I\subseteq  \K[x]$ possesses a
finite Pommaret basis for a term order $\prec$ in 
suitably chosen coordinate systems.
\end{theorem}

\begin{definition}\cite[Definition 3.4.2]{Seiler2010}
Let $\mathcal{F}\subset \K[x]\backslash\{0\}$
be a finite set of polynomials and $L$ be an involutive division on
$\N^n$. We assign to each element $f \in \mathcal{F}$ a set of
multiplicative variables
\[
X_{L,\mathcal{F},\prec}(f)=\{x_i \mid i\in N_{L,\lee_{\prec}
\,\mathcal{F}}(\lee_{\prec} f)\}.
\]
The involutive span of $\mathcal{F}$ is then the set
\[
\langle \mathcal{F} \rangle_{L,\prec} =\sum_{f\in \mathcal{F}}
\K[X_{L,\mathcal{F},\prec}(f)]\cdot  f\subseteq \langle \mathcal{F}
\rangle.
\]
\end{definition}

\begin{theorem}\cite[Theorem 3.4.4]{Seiler2010}\label{t14}
Let $I \subseteq  \K[x]$
be 
a nonzero ideal, $\mathcal{H}\subset I\backslash\{0\}$ a finite set
and $L$ an involutive division on $\N^{n}.$
 Then the following two statements are equivalent. \\
\begin{itemize}
\item [(i)]  The set $\mathcal{H}\subset I$ is a weak involutive basis of $I$ with respect to $L$ and $\prec$.\\
\item [(ii)]  Every polynomial $f \in I$ can be written in the form
\begin{equation}\label{strong involutive basis}
f=\sum_{h \in \mathcal{H} } P_h \cdot h 
\end{equation}
with coefficients $P_h \in \K[X_{L,\mathcal{H},\prec}(h)]$
satisfying $\lt(P_h \cdot h)\preceq \lt(f)$ for all polynomials
$h\in \mathcal{H} $ such that $P_h\neq 0$.
\end{itemize}
$\mathcal{H}$ is a strong involutive basis, if and only if the
representation (\ref{strong involutive basis}) is unique.
\end{theorem}

\begin{corollary}\cite[Corollary 3.4.5]{Seiler2010}\label{t16}
Let $\mathcal{H}$ be a weak involutive basis of the ideal
$I \subseteq \K[x]$.
 Then $\langle\mathcal{H}\rangle_{L,\prec}=I$.
If $\mathcal{H}$ is even a strong involutive basis of $I$, then $I$
considered as a $\K$-linear space possesses a direct sum
decomposition $I=\bigoplus_{h \in
\mathcal{H}}\K[X_{L,\mathcal{H},\prec}(h)] \cdot h$.
\end{corollary}

\begin{proposition}\cite[Proposition
3.4.7]{Seiler2010}\label{stronginvolutive} Let $I\subseteq\K[x]$ be
an ideal and $\mathcal{H}\subset I$ be a weak involutive basis of
$I$ for the involutive division $L$. Then there exists a subset
$\mathcal{H}^{'}\subseteq\mathcal{H}$ which is a strong involutive
basis of I.
\end{proposition}

%

\begin{definition}
If we regard $\K[x]$ as a linear space, then the ideal $I$ and the
truncated ideal $I_t = I\cap \K[x]_t$ are both subspaces in $\K[x]$.
We say that the set $G=\{g_1,\ldots,g_s\}$
 is a reduced basis of  $I_t$, if it is a linear independent basis of $I_t$ and
 all polynomials in $G$ have different leading monomials with respect to
 a given term order.
\end{definition}

\section{computing a pommaret basis
}\label{mainresults}
 In this section, we present an algorithm  as well as a certificate
 for
computing a Pommaret basis for an ideal $J$, s.t.
 $I \subseteq J \subseteq I(V_{\R}(I))$ when $V_{\R}(I)$ is positive
 dimensional. The certificate given in (\ref{e,5}) generalizes the flat extension
 conditions in \cite{LLR08}
  for the zero-dimensional real variety  to the
 positive dimensional case.

\subsection{The certificate }
Let $I=\langle h_1, \ldots,h_m\rangle \subseteq \R[x]$ be an ideal
and $d:=\max_{1\leq j\leq m}d_j$, $d_j:=\lceil\deg(h_j)/2\rceil$.
 For each $t\geq d$, recall the notions

\begin{equation*}
    \Kk_t:=\{y\in\R^{\N^n_{2t}}\mid y_0=1, M_t(y)\succeq0, M_{t-d_j}(h_jy)=0,
    j=1,\ldots,m\},
\end{equation*}
and
\begin{equation*}
\Kk_t^{gen}:=\{ y \in \Kk_t \mid \rank  \, M_t(y) ~{\rm is ~maximum~
over}~ \Kk_t\}.
\end{equation*}

For a moment matrix $M_t(y)$ of order $t$, the truncated moment
matrix $M_{t-\ell}(y)$ for $\ell < t$ is the order $t-\ell$
principal submatrix of $M_t(y)$ indexed by $\alpha,
\beta\in\N_{t-\ell}^n.$

Let $\alpha_j$ denote the number of class $j$ polynomials of degree
$t-2$ in the reduced basis of $\ker M_{t-2}(y)$.
 We have the
following theorem:

\begin{theorem}\label{maintheorem}
Let $I=\langle h_1,\ldots,h_m\rangle\subseteq\R[x]$
be an ideal. If
there exists an integer  $t\geq 2d$ satisfying
\begin{equation}
\label{e,5} \sum_{j=1}^n j\alpha_{j} ~\text{for} ~\ker M_{t-2}(y)=
\crk \, M_{t-1}(y)-\crk \, M_{t-2}(y),
\end{equation}
for   an element  $y\in \Kk_t^{gen}$.  Then a reduced basis of the
null space of $ M_{t-2}(y)$ is a weak Pommaret basis for $J=\langle
\ker M_{t-2}(y) \rangle$ under the monomial ordering $\tdeg$ and
\begin{equation}
 I\subseteq \langle J \rangle \subseteq I(V_{\R}(I)), ~V_\R(I)=V_\C(J) \cap
 \R^n.
\end{equation}
\end{theorem}
The proof of Theorem \ref{maintheorem} follows from  Proposition
\ref{t2} and  Theorem \ref{t8} whose proofs are given in Section
\ref{justification}.

\begin{remark}\label{t11}
Although the reduced bases of $\ker M_{t-2}(y)$ are not unique, they
have the same set of leading terms since they can be represented
linearly by each other. Therefore, each reduced basis of $\ker
M_{t-2}(y)$ has the same value of  $\sum_{j=1}^n j\alpha_{j}$.
\end{remark}

In our algorithm, we need to find an element  $y$ in $\Kk_t$
maximizing the rank of $M_t(y)$. As pointed out in \cite{LLR08},
this could be done typically by solving the semidefinite program
\begin{align}\label{e,11}
\min \quad 0 \quad {\text s.t.} \quad y \in \Kk_t
\end{align}
with interior-point algorithms using self-dual embedding, see
\cite{VB96,WSV00}.

\subsection{An algorithm for computing a Pommaret basis}
We list main steps of our algorithm based on solving (\ref{e,11})
for computing a Pommaret basis of the ideal $J= \langle \ker
M_{t-2}(y) \rangle$  nested between $I$ and $I(V_\R(I))$.

\begin{algorithm}\label{e,6} Computing a Pommaret basis of an ideal
$J$ such that
 $I \subseteq J \subseteq I(V_{\R}(I))$.

\textbf{Input:} A set of polynomials $\{h_1,\ldots,h_m\}$ generating
$I$ and the monomial ordering
 $\tdeg$ on variables $x_1,\ldots,x_n$.

\textbf{Output:} A Pommaret basis for $\langle\ker
M_{t-2}(y)\rangle$ under the monomial ordering $\tdeg$.

\begin{description}

    \item [Step~1] For $t\geq 2d$,  compute  a generic element $y\in\Kk_t$ by solving
    (\ref{e,11}).

    \item [Step~2] Compute a reduced basis of $\ker M_{t-1}(y)$. Let $\{g_1, \ldots, g_{s+t}\}$ be   polynomials
    of degree $t-2$ 
    in this reduced basis.  Compute the  value of $\sum_{j=1}^n j\alpha_{j}$. 

    \item [Step~3] 
Compute $\crk \, M_{t-1}(y)-\crk \, M_{t-2}(y)$ by calculating the
number of polynomials of degree $t-1$ in the reduced basis of $\ker
M_{t-1}(y)$. 

    \item [Step~4]  Test whether the condition (\ref{e,5}) is satisfied.

    \begin{itemize}

    \item  If yes, $\{g_1, \ldots, g_{s+r}\}$ is
 a weak  Pommaret basis for $\langle \ker M_{t-2}(y)
    \rangle$ and can be  
 reduced further to a (strong) Pommaret basis.  
%
%

   \item  Otherwise, let $t:=t+1$ and go to Step 1.
   \end{itemize}
\end{description}
\end{algorithm}

In  Section \ref{justification}, we prove that Algorithm \ref{e,6}
is correct and terminates in a finite number of steps in a $\delta$-regular coordinate system for $\sqrt[\R]{I}$.
 The algorithm  has been implemented in Matlab
using the GloptiPoly toolbox \cite{gloptipoly} and we demonstrate
its performance on a set of  examples  in Section \ref{examples}.

\begin{remark}
In order to check the condition (\ref{e,5}), we need to compute a
reduced basis of the  null space of the truncated moment matrix
$M_{t-1}(y)$.
These computations have to be performed stably. 
For the computation of a reduced basis 
it is 
important to choose a proper tolerance to ensure that there is no
information missing in  $\ker M_{t-1}(y)$. We list the tolerance
used for each example in Section 4.
\end{remark}

\subsection{Justification of the certificate}\label{justification}

Our main goal in this section is to prove that Algorithm \ref{e,6}
is correct and it terminates after a finite number of steps in a $\delta$-regular coordinate system for $\sqrt[\R]{I}$.

\begin{assumption}\label{assump}
Let $I=\langle h_1,\ldots,h_m\rangle\subseteq\R[x]$ be an ideal.
Suppose there exists an integer  $t\geq 2d$
satisfying the condition
(\ref{e,5}) for  $y\in \Kk_t^{gen}$. Let \{$g_1,\ldots,g_{s+r}$\} be
a reduced basis of  $\ker M_{t-2}(y)$, where
\[ \deg(g_i)=t-2 ~{\it for}~ 1\leq i \leq s,
  ~{\it and }~ \deg(g_i)<t-2 ~{\it for}~ s+1\leq i \leq s+r.\]
%
\end{assumption}

\begin{lemma}\label{t7}
Under Assumption \ref{assump}, the polynomial set
\[
\{x_1g_1,\ldots,x_{j_1}g_1,\ldots,x_1g_s,\ldots,x_{j_s}g_s,g_1,\ldots,g_{s+r}\}
\]
is a reduced basis of $\ker M_{t-1}(y)$, where   $j_i=\cls (g_i)$
for $i=1,\ldots,s$.
\end{lemma}

\begin{proof}
For $k=1,\ldots,n$, $i=1,\ldots,s+r$, since $\deg(x_kg_i)\leq t-1$,
by Proposition \ref{t1} (i), we have $x_kg_i\in\ker M_{t-1}(y)$.
 In
fact, since each polynomial in $\{g_1,\ldots,g_{s+r}\}$ has
different leading terms, according to Definition \ref{indef},
 the polynomials
\begin{equation}\label{xigj} x_1 g_1, \ldots,
x_{j_1} g_1, \ldots, x_1 g_s, \ldots, x_{j_s} g_s
\end{equation}
all have distinct leading terms of degree $t-1$. Hence they are
linearly independent.  Suppose there are $\alpha_{j}$ polynomials of
class $j$ in $\{g_1,\ldots,g_{s}\}$, then polynomials in
(\ref{xigj})
 yield $\sum_{j=1}^n j\alpha_{j}$ linearly independent polynomials
of degree $t-1$ in $\ker M_{t-1}(y)$. On the other hand, the number
of linearly independent polynomials of degree $t-1$ in a reduced basis of
$\ker M_{t-1}(y)$ equals to $\crk \, M_{t-1}(y)-\crk \,
M_{t-2}(y)$. Hence,  the condition
(\ref{e,5}) and Proposition \ref{t1} (iii)
 implies that the conclusion is
true.
\end{proof}

\begin{remark}\label{t15}
Under Assumption \ref{assump}, for any polynomial $f\in\ker
M_{t-1}(y)$,
we can express it as a linear
combination:
\begin{equation}\label{rem4.5}
f=\sum_{k=1}^{s}\sum_{i=1}^{\cls(g_{k})}c_{ik}x_ig_k+\sum_{k=1}^{s+r}\lambda_{k}g_{k},
\end{equation}
where $c_{ik}\in\R$  and
 $\lt(c_{ik}x_ig_k) \ptdeg \lt(f)$
for $1 \leq i \leq \cls(g_{k})$ and $ 1 \leq k \leq s$,
$\lambda_k\in\R$ and  $\lt( \lambda_k g_k) \ptdeg \lt(f)$ for $ 1
\leq k \leq s+r$. Note that every polynomial in
$\{x_1g_1,\ldots,x_{j_1}g_1,\ldots,x_1g_s,\ldots,x_{j_s}g_s,g_1,\ldots,g_{s+r}\}$
has a different leading term. Under the graded monomial ordering
$\tdeg$, there is only one $c_{i_0k_0}\neq0$ with
$\lt(x_{i_0}g_{k_0})=\lt(f)$  if not all $c_{ik}$ are zeros.
 This property is very important
and will be used in the proofs of  theorems below.
\end{remark}

\begin{lemma}\label{t10}
Under Assumption \ref{assump}, for all monomial $x^\mu$ and
polynomials $ g_{j}$ with $\deg(g_j)<t-2$, $j=s+1,\ldots,s+r$,
 the polynomial $ x^\mu g_{j}$ can be expressed as
\begin{equation}\label{e,8} x^\mu
g_{j}=\sum_{k=1}^{s}h_{k}g_k+\sum_{k=s+1}^{s+r}{\lambda_{k}}g_{k},
\end{equation}
where  $h_{k}\in\R[x]$ and $\lambda_k \in \R $ satisfying
$\lt(h_{k}g_k)\preceq_{\text{\upshape tdeg}} \lt(x^\mu
g_{j}), \ k=1,\ldots,
s$ and $\lt(\lambda_k g_k)\preceq_{\text{\upshape tdeg}} \lt(x^\mu
g_{j}),\
k=s+1,\ldots,s+r$.
\end{lemma}

\begin{proof}
If $\deg(x^{\mu}  g_j)  \leq t-1$, by Proposition \ref{t1} (i), we
have $x^{\mu} g_j\in\ker M_{t-1}(y)$. According  to Remark
\ref{t15}, we have the expression (\ref{e,8}).  Otherwise,
 we set
$x^{\mu}=x^{\mu_{1}} x^{\mu_{2}}$  such that $\deg(x^{\mu_2}
g_j)=t-1$. Hence,  we have
\begin{align*}
x^{\mu}g_j &=x^{\mu_1}x^{\mu_2} g_{j} =
   x^{\mu_1}(\sum_{k=1}^{s}h_{k}g_k+\sum_{k=s+1}^{s+r}{\lambda_{k}}g_{k})\\
  &=\sum_{k=1}^{s}
  x^{\mu_1} h_k
  g_k+\sum_{k=s+1}^{s+r}\lambda_{k}x^{\mu_1}g_{k}.
\end{align*}
We can repeat the above reduction on $x^{\mu_1}g_{k}$ for $s+1 \leq
k \leq s+r$. Since $\deg(x^{\mu_1})<\deg(x^{\mu})$, after a finite
number of steps, we have the expected form (\ref{e,8}).
\end{proof}

\begin{theorem}\label{t8}
 Under Assumption \ref{assump},
a reduced basis $\{g_1,\ldots,g_{s+r}\}$ of  $\ker M_{t-2}(y)$ is a
weak Pommaret basis of the ideal $\langle \ker M_{t-2}(y) \rangle$.
\end{theorem}

\begin{proof}
We  show  that  any polynomial  
 $f\in \langle \ker M_{t-2}(y) \rangle$ can be represented as
\begin{equation}\label{e,7}
f=\sum_{k=1}^{s}h_{k}g_k+\sum_{k=s+1}^{s+r}{\lambda_{k}}g_{k},
\end{equation}
where
 $\lambda_{k}\in \R$ and $h_{k}\in\R[x_1,\ldots,x_{\cls(g_{k})}]$.  Since $\lt(h_{k}g_k)$ and $\lt(g_k)$ are all
different for $1 \leq k \leq s+r$, if
 $f$ satisfies (\ref{e,7}), then we have
$\lt(h_{k}g_k)\preceq_{\text{\upshape tdeg}} \lt(f)$ for $1 \leq k
\leq s$ and $\lt(        
\lambda_k g_k) \preceq_{\text{\upshape tdeg}} \lt(f)$ for $s+1 \leq
k \leq s+r$. Therefore, according to Theorem \ref{t14}, the
polynomial set $\{g_1,\ldots,g_{s+r}\}$ is a weak Pommaret basis of
the ideal $\langle \ker M_{t-2}(y) \rangle$.

Since $\{g_1,\ldots,g_{s+r}\}$ is a reduced basis of $\ker
M_{t-2}(y)$,  every
 polynomial $f  \in    \langle \ker M_{t-2}(y)  \rangle $ can be
represented as
\[f=\sum_{j=1}^{s+r}h_{j}g_{j},\]
where $h_j\in\R[x], ~j=1,\ldots, s+r$.  Hence, we only
need to show that each polynomial $x^{\mu}g_j$ for $ \mu \in \N^{n}$ and
$ 1 \leq j \leq s+r$ can be
written as (\ref{e,7}).

Set $f=x^{\mu}g_j$.
If  $\deg(f) \leq t-1$, by Lemma \ref{t7}, we have the expected
expression (\ref{e,7}) directly. 
Otherwise, we prove
by the induction on its leading term $\lt(f)= t_0$, i.e., 
we assume that $f=x^{\mu}g_j$ has the expected expression
(\ref{e,7}) as long as $\lt(f) \tdeg t_0$ for $ \mu \in \N^{n}$ and
$1 \leq j \leq s+r$, we show it has the expected expression when
$\lt(f)= t_0$.

 If $x^{\mu}\in \R[x_1,\ldots,x_{\cls(g_{j})}]$, nothing
is to be proved. Otherwise, without loss of generality, let
$x_{i_1}$  be a non-multiplicative variable in $x^{\mu}$ with
respect to $g_j$. Since $\deg(g_j)\leq t-2$, $j=1,
\ldots, s+r$, by Proposition \ref{t1} 
(i),
we have $ x_{i_1} g_{j} \in \ker M_{t-1}(y)$. By Lemma \ref{t7} and
Remark \ref{t15},  we have
\begin{align}\label{eq1}
f&=x^{\mu}g_{j}=\big(x^{\mu}/x_{i_1}\big)~ x_{i_1} g_{j}\nonumber\\
  &=\big(x^{\mu}/x_{i_1}\big)\big(\sum_{k=1}^{s}\sum_{i=1}^{\cls(g_k)}
  c_{ik}x_ig_k+\sum_{k=1}^{s+r}\lambda_{k}g_{k}\big) \nonumber \\
 &=\sum_{k=1}^{s}\sum_{i=1}^{\cls(g_k)}c_{ik}\big(x^{\mu}/x_{i_1}\big) x_ig_k+\sum_{k=1}^{s+r}\lambda_{k} \big(x^{\mu}/x_{i_1}\big) g_{k}.
\end{align}
 According to Remark \ref{t15},  there are  two cases:
\begin{itemize}
\item [(i)] if all  $c_{ik}=0$,
there exits
only one $1\leq j_1\leq s+r$, such that $\lambda_{j_1}\neq0$ and
$\lt(\lambda_{j_1}\big(x^{\mu}/x_{i_1}\big)g_{j_1})=t_0$;

\item [(ii)]otherwise, there exists
$1\leq j_1\leq s$ and $1\leq i_2\leq \cls(g_{j_1})$  such that
$c_{i_{2}j_{1}}\neq0$ and
$\lt(c_{i_{2}j_{1}}\big(x^{\mu}/x_{i_1}\big)x_{i_2}g_{j_1})=t_0$.
\end{itemize}
In both cases, 
all  other terms in (\ref{eq1}) have leading terms of order less
than $t_0$, which can be expressed as (\ref{e,7}) by induction.
Moreover,   above two cases do not exist simultaneously.
  Therefore, we only need
 to check whether the polynomial
$\lambda_{j_1}\big(x^{\mu}/x_{i_1}\big)g_{j_1}$ in case (i) or
$c_{i_{2}j_{1}}\big(x^{\mu}/x_{i_1}\big) x_{i_2}g_{j_1}$ in case
(ii) has the representation (\ref{e,7}).
%

In case (i),  if $x^{\mu}/x_{i_1} \in \R[x_1,\ldots,x_{\cls(g_{j_1})}]$
then we  obtain  the representation (\ref{e,7}). Otherwise, we
repeat the reduction to the polynomial $\big(x^{\mu}/x_{i_1}\big)
g_{j_1}$.  Since
$\lt(\lambda_{j_1}\big(x^{\mu}/x_{i_1}\big)g_{j_1})=\lt(x^{\mu}g_{j})=t_0$,
we have $\deg(g_j)<\deg(g_{j_1})$, i.e.,
\[\lt(g_{j})\tdeg\lt(g_{j_1}).\]

In case (ii),   if $x^{\mu}/x_{i_1} \in
\R[x_1,\ldots,x_{\cls(g_{j_1})}]$, since $x_{i_2}$ is a multiplicative
variable of $\lt(g_{j_1})$,  then
$\big(x^{\mu}/x_{i_1}\big)x_{i_2} \in \R[x_1,\ldots,x_{\cls(g_{j_1})}]$.
Hence,  we  obtain  the representation (\ref{e,7}). Otherwise, since
$x_{i_1}$ is a non-multiplicative variable of $\lt(g_{j})$ and
$x_{i_2}$ is a multiplicative variable of $\lt(g_{j_1})$, we have
\begin{align*}
\cls(g_{j})< \cls(x_{i_1}), \quad \cls(x_{i_2}) \leq
~\cls(g_{j_1}).
\end{align*}
Because $\lt(c_{i_{2}j_{1}}\big(x^{\mu}/x_{i_1}\big)
x_{i_2}g_{j_1})=
t_0$, we have  $\lt(x_{i_2}g_{j_1})=\lt(x_{i_1}
g_{j})$ and
\begin{align}\label{eq2}
&\cls(x_{i_2})=\cls(x_{i_2}g_{j_1})=\cls(x_{i_1}g_{j})<\cls(x_{i_1}).
\end{align}
This implies that $x_{i_2}\tdeg x_{i_1}$. If $\lt(g_{j_1})
\preceq_{\text{\upshape tdeg}}
\lt(g_{j})$,
we have $\lt(x_{i_2}g_{j_1})$ $\tdeg$ $\lt(x_{i_1} g_{j})$ which
leads to
a contradiction. Therefore, we can deduce that
\[\lt(g_{j})\tdeg \lt(g_{j_1}).\]

In both cases, if the reduction does not stop, we will obtain a
sequence of polynomials satisfying
\[\lt(g_j) \tdeg \lt(g_{j_1}) \tdeg \cdots \tdeg  \lt(g_{j_{i}}) \tdeg
\lt(g_{j_{i+1}}) \tdeg \cdots \tdeg t_0.\]
 Since the number of polynomials with strict increase leading terms
  bounded by $\lt(f)=t_0$ is finite, the above procedure
 will stop in a finite number of steps  and we obtain the expected form (\ref{e,7}) for
 $f$.
\end{proof}

\begin{theorem}\label{t12}
In a $\delta$-regular coordinate system for $\sqrt[\R]{I}$, after a
finite number of  steps, Algorithm \ref{e,6} will terminate and
return an integer  $t\geq 2d$ which satisfies the condition
(\ref{e,5}) for an element $y\in \Kk_t^{gen}$.
\end{theorem}

\begin{proof}
In a $\delta$-regular coordinate system, we have a finite  Pommaret
basis $\mathcal{H}=\{h_1,\ldots,h_s\}$ for the real radical ideal
$I(V_{\R}(I))$. According to Proposition \ref{t2} (iii), we can
conclude that there exists an integer $t_1 $ such that
 the Pommaret basis $\{h_1,\ldots,h_s\}$ is contained in
$\ker M_{t}(y)$ for all $y \in \Kk_t$ and $t\geq t_1$.

Since $\mathcal{H}$ is a  Pommaret basis of $I(V_{\R}(I))$,
according to Corollary \ref{t16}, for $ t\geq t_1+2$, we
have the following decomposition: 
\begin{equation}\label{e,12}
I(V_{\R}(I))_{t-2}=\bigoplus_{h_k \in
\mathcal{H}}\R[x_1,\ldots,x_{\cls(h_k)}]_{t-2-\deg(h_k)} \cdot h_k.
\end{equation}
Let
\begin{equation}\label{Teq}
T=\{x^u h_k\mid x^u\in \R[x_1,\ldots,x_{\cls(h_k)}], ~\deg(x^u)\leq
t-2-\deg(h_k), 1\leq k\leq s\}.
\end{equation}
According to  Proposition \ref{t1} (i),
$T\subseteq\ker M_{t-2}(y)$.
 Therefore, by (\ref{e,12}) and
(\ref{Teq}), we have
\[I(V_{\R}(I))_{t-2} \subseteq \ker
M_{t-2}(y).\] On the other hand,  $y$ is a generic element, by
Proposition \ref{t2} (i), we have
\[\ker M_{t-2}(y)\subseteq
I(V_{\R}(I))_{t-2}.\]
 Hence, we have
 $\ker
M_{t-2}(y)=I(V_{\R}(I))_{t-2}$  and the  decomposition:
\begin{equation}\label{Teq2} \ker M_{t-2}(y)=\bigoplus_{h_k \in
\mathcal{H}}\R[x_1,\ldots,x_{\cls(h_k)}]_{t-2-\deg(h_k)} \cdot h_k.
\end{equation}

 Since $\mathcal{H}$ is a  Pommaret basis of $I(V_{\R}(I))$, according to Definition \ref{f,2},
each polynomial in $T$ has a
different leading term.  Therefore $T$ is actually  a reduced basis
of  $\ker M_{t-2}(y)$. By Remark \ref{t11}, it suffices to show that
the condition (\ref{e,5}) holds for the polynomials in $T$.

Similar to the decomposition (\ref{Teq2}), we can show that  there
exists a  direct sum decomposition of $\ker M_{t-1}(y)$:
\begin{equation}\label{Teq3}
\ker M_{t-1}(y)=\bigoplus_{h_{k} \in
\mathcal{H}}\R[x_1,\ldots,x_{\cls(h_k)}]_{t-1-\deg(h_{k})} \cdot
h_{k}.
\end{equation}
 For a polynomial $f \in \ker M_{t-1}(y)$  with  $
\deg(f)=t-1$, according to (\ref{Teq3}), 
we have the
following equalities:
\begin{align*}
f&=\sum_{k=1}^{s}\sum_{0\leq|\mu|\leq t-1-\deg(h_k)}c_{\mu k}x^{\mu}h_k   ~~(\text{note that}~ x^{\mu} \in \R[x_1,\ldots,x_{\cls(h_k)}]) \\
 &=\sum_{k=1}^{s}\sum_{|\mu|=t-1-\deg(h_k)}c_{\mu k}x^{\mu}h_k+\sum_{k=1}^{s}\sum_{0\leq|\mu|\leq t-2-\deg(h_k)}c_{\mu k}x^{\mu}h_k\\
 &=\sum_{k=1}^{s}\sum_{|\mu|=t-1-\deg(h_k)}c_{\mu k}
 x_{\cls(x^{\mu})}\big(x^{\mu}/x_{\cls(x^{\mu})}\big)
  h_k+\sum_{k=1}^{s}\sum_{0\leq|\mu|\leq t-2-\deg(h_k)}c_{\mu k}x^{\mu}h_k.
\end{align*}
Since $x_{\cls(x^{\mu})}$ is always a multiplicative variable for
the polynomial $\big(x^{\mu}/x_{\cls(x^{\mu})}\big)h_k\in T$, we
know that each polynomial in $\ker M_{t-1}(y)$ can be represented by
the polynomials in  $T$ and $T_1$, where
\[
T_1=\{x_ig\mid 1\leq i\leq\cls(g),\ g \in T, 
\deg(g)=t-2\}.
\]
The polynomials in $T_1$ and $T$  have different leading terms,
hence $T \cup T_1$ is   a linearly independent basis of 
 $\ker
M_{t-1}(y)$. Moreover, $T$ is  a reduced
 basis of $\ker M_{t-2}(y)$,
   and $T_1$ consists of all linearly independent polynomials with degree $t-1$ in $\ker M_{t-1}(y)$. 
We can deduce that the number of polynomials in $T_1$ is equal to
$\crk \, M_{t-1}(y)-\crk \, M_{t-2}(y)$.  On the other hand,
 let  $\alpha_{j}$ denote the
number of polynomials of class $j$ and degree $t-2$ in $T$. Since
the set $T_1$  is constructed by multiplying   polynomials in $T$ of
degree $t-2$  by their  multiplicative variables only, the total
number of polynomials in $T_1$ is equal to $\sum_{j=1}^{n}
j\alpha_{j}.$
 Therefore, the condition (\ref{e,5}) is satisfied.
\end{proof}

\subsection{An Extension to  $I(V_{\R}(I) \cap \Aa)$} Consider the
semialgebraic set $\Aa:=\{x\in\R^n\mid f_1(x)\geq0, \ldots,
f_s(x)\geq0\}$ and  the $\Aa$-radical ideal $I(V_{\R}(I)\cap \Aa)$.
We  restrict to a subset $\Kk_{t,\Aa}\subseteq\Kk_{t}$ defined as
\begin{align*}
    \Kk_{t,\Aa}:= \Kk_t \cap \left\{y\in\R^{\N^n_{2t}}\mid
    M_{t-d_{f^{\nu}}}(f^{\nu} y)\succeq0\,\,\text{for
    all}\,\nu\in\{0,1\}^s \right\},
\end{align*}
where $d_{f^{\nu}}=\lceil\deg(f^{\nu})/2\rceil$ and $t \geq d=
\max_{1\leq j\leq m, {\nu}\in\{0,1\}^s} \{d_j,~
     d_{f^{\nu}}\}$.

For $t$ large enough, Lemma \ref{t5} and Theorem \ref{t6} show that
the information about $\sqrt[\Aa]{I}$ is contained in the projection
of a generic element $y\in\Kk_{t,\Aa}$. Thus, propositions and
theorems discussed above are true for generic elements $y$ in
$\Kk_{t,\Aa}$.

 The following theorem can be seen as a variant
of Theorem \ref{maintheorem} for the semialgebraic set $\Aa$. The
proof uses exactly the same reason as in Theorem \ref{t8} and
Theorem \ref{t12}   after replacing $\Kk_t$ and
 $\sqrt[\R]{I}$ by $\Kk_{t,\Aa}$ and $\sqrt[\Aa]{I}$ respectively.

\begin{theorem}
Suppose the condition (\ref{e,5})  holds for a generic element
$y\in\Kk_{t,\Aa}$, and $t \geq 2d$. Then a reduced basis of the null
space of $M_{t-2}(y)$
 is a weak Pommaret basis of
 $\langle \ker
M_{t-2}(y) \rangle$ under the monomial ordering $\tdeg$ and
\begin{equation*}
I \subseteq \langle \ker M_{t-2}(y) \rangle \subseteq I(V_{\R}(I)
\cap \Aa).
\end{equation*}
\end{theorem}

\begin{remark}
For computing a Pommaret basis of $\langle \ker M_{t-2}(y) \rangle$,
we add the defining polynomials $\{f_1,\ldots,f_s\}$ of the
semialgebraic set $\Aa$ to the input of the above algorithm and
additional constraints $M_{t-d_{f^{\nu}}}(f^{\nu}
y)\succeq0\,\,\text{for
    all}
    ~\nu\in\{0,1\}^s$ to the semidefinite program (\ref{e,11}).
\end{remark}

\section{Numerical examples}\label{examples}

We present here the results obtained by  applying Algorithm
\ref{e,6} to some examples in
\cite{Rostalski09,SRWZ09,Seiler02,Stetter04} and others.

\begin{example}
Consider the 2-dimensional ideal $I=\langle h_1, h_2, h_3\rangle$
taken from \cite[p.397, Eq. (9.60)]{Stetter04} where
\begin{align*}
    h_1 &= x_1^2+x_1x_2-x_1x_3-x_1-x_2+x_3,\\
    h_2 &= x_1x_2+x_2^2-x_2x_3-x_1-x_2+x_3,\\
    h_3 &= x_1x_3+x_2x_3-x_3^2-x_1-x_2+x_3.
\end{align*}
The rank and  corank sequences for truncated moment matrices
$M_{t-\ell}(y)$ are shown in Table 1 and 2. We set $\tau=10^{-5}$
and $x_1\tdeg x_2\tdeg x_3$. For  t=4,  we have
\[\sum_{j=1}^3 j\alpha_{j}=6, ~{\text{and}}~ 
\crk \, M_{4-1}-\crk \, M_{4-2}=6.\]
Hence, the condition
(\ref{e,5}) is satisfied. The Pommaret basis  computed by
Algorithm \ref{e,6} for $t=4$ is
\begin{align*}
\{x_1+x_2-x_3-x_1^2-x_1x_2+x_1 x_3, x_1+x_2-x_3-x_1 x_2-x_2^2+x_2
x_3, \\
 3 x_1+3 x_2 -3 x_3 -x_1^2-2 x_1 x_2-x_2^2+x_3^2\}.
 \end{align*}
 From Table 3, we note that the condition (\ref{e,5}) is also
satisfied for  $t=5,6,7$.  For this example, we can show that
$\langle\ker M_{4-2}(y)\rangle=\sqrt[\R]{I}$, and a reduced   basis of $\ker
M_{4-2}(y)$  is a Pommaret basis of $\sqrt[\R]{I}$. Hence,
 the condition (\ref{e,5}) can
be satisfied by arbitrary  $t \geq 4$.
\begin{table*}[!htb]
\begin{minipage}{0.48\textwidth}
\begin{center}
\caption{\small The rank of $M_{t-\ell}(y)$ }
\begin{tabular}{cccccccc}
  \hline\noalign{\smallskip}
Order & $\ell=0$ & $\ell=1$ &$ \ell=2$ \\
\noalign{\smallskip}\hline\noalign{\smallskip}
   t=4  & 16& 11& 7 \\
   t=5  & 22& 16& 11\\
   t=6  & 29& 22& 16\\
   t=7  & 37& 29& 22\\
   \noalign{\smallskip}\hline
\end{tabular}
\end{center}
\end{minipage}
\begin{minipage}{0.5\textwidth}
\begin{center}
\caption{\small The corank of $M_{t-\ell}(y)$ }
\begin{tabular}{cccccccc}
  \hline\noalign{\smallskip}
Order & $\ell=0$ & $\ell=1$ &$ \ell=2$ \\
\noalign{\smallskip}\hline\noalign{\smallskip}
   t=4  & 19&  9& 3 \\
   t=5  & 34& 19& 9 \\
   t=6  & 55& 34& 19\\
   t=7  & 83& 55& 34\\
   \noalign{\smallskip}\hline
\end{tabular}
\end{center}
\end{minipage}
\end{table*}
\\
\begin{table*}[!htb]
\center \caption{\small The $\alpha_j$ of a reduced basis of $\ker
M_{t-2}(y)$}
\begin{tabular}{ccccccccc}
   \hline\noalign{\smallskip}
Order & $\alpha_1$ &$ \alpha_2$&$ \alpha_3$ &$\sum_{j=1}^3 j\alpha_{j}$ \\
\noalign{\smallskip}\hline\noalign{\smallskip}
   t=4  & 1 & 1 & 1 & 1$\times$1+1$\times$2+1$\times$3=6&\\
   t=5  & 3 & 2 & 1 & 3$\times$1+2$\times$2+1$\times$3=10&\\
   t=6  & 6 & 3 & 1 & 6$\times$1+3$\times$2+1$\times$3=15&\\
   t=7  & 10& 4 & 1 & 10$\times$1+4$\times$2+1$\times$3=21&\\
   \noalign{\smallskip}\hline
\end{tabular}
\end{table*}

\end{example}

\begin{example}
Consider the polynomial system $P=\{h_1, h_2\}$ in \cite[p.20, Ex
1.4.6]{SRWZ09}, where
\begin{align*}
    h_1 &= x_1^2-x_2,\\
    h_2 &= x_1x_2-x_3.
\end{align*}
For the term order   $x_3\tdeg x_1\tdeg x_2$, we have $\cls(x_1)=2$,
$ \cls(x_2)=3$, $ \cls(x_3)=1$.  Let  $\tau=10^{-8}$,  the rank and
corank sequences for truncated moment matrices $M_{t-\ell}(y)$ are
shown in  Table 4 and 5.
\begin{table*}[!htb]
\begin{minipage}{0.47\textwidth}
\begin{center}
\caption{\small The rank of $M_{t-\ell}(y)$ }
\begin{tabular}{cccccccc}
  \hline\noalign{\smallskip}
Order & $\ell=0$ & $\ell=1$ &$ \ell=2$ \\
\noalign{\smallskip}\hline\noalign{\smallskip}
   t=3  & 12&  7& 4 \\
   t=4  & 16& 10& 7 \\
   t=5  & 20& 13& 10 \\
   \noalign{\smallskip}\hline
\end{tabular}
\end{center}
\end{minipage}
\begin{minipage}{0.5\textwidth}
\begin{center}
\caption{\small The corank of $M_{t-\ell}(y)$ }
\begin{tabular}{cccccccc}
  \hline\noalign{\smallskip}
Order & $\ell=0$ & $\ell=1$ &$ \ell=2$ \\
\noalign{\smallskip}\hline\noalign{\smallskip}
   t=3  &  8&  3& 0 \\
   t=4  & 19& 10& 3 \\
   t=5  & 36& 22& 10 \\
   \noalign{\smallskip}\hline
\end{tabular}
\end{center}
\end{minipage}
\end{table*}

\begin{table*}[!htb]
\center \caption{\small  The  $\alpha_j$ of a reduced basis of $\ker
M_{t-2}(y)$}
\begin{tabular}{ccccccccc}
   \hline\noalign{\smallskip}
Order & $\alpha_1$ &$ \alpha_2$& $\alpha_3$ &$\sum_{j=1}^3 j\alpha_{j}$ \\
\noalign{\smallskip}\hline\noalign{\smallskip}
   t=3  & 0 & 0 & 0 & 0$\times$2+0$\times$3+0$\times$1=0\\
   t=4  & 2 & 1 & 0 & 2$\times$2+1$\times$3+0$\times$1=7\\
   t=5  & 3 & 1 & 3 & 3$\times$2+1$\times$3+3$\times$1=12\\
   \noalign{\smallskip}\hline
\end{tabular}
\end{table*}

For t=4, we have
\[\sum_{j=1}^3 j\alpha_{j}=7, ~{\text{and}}~
\crk \, M_{4-1}-\crk \, M_{4-2}=7.\] Hence,   the condition
(\ref{e,5}) is satisfied. The Pommaret basis  computed by
Algorithm \ref{e,6} for $t=4$ is
\[
\{x_1^2-x_2, x_1 x_2 -x_3, x_2^2-x_1 x_3\}.
\]

\end{example}

\begin{example}
Consider the ideal $I=\langle h_1, h_2\rangle$ in \cite[p.123, Ex
7.41]{Rostalski09} with
\begin{align*}
    h_1&=x_1^2+x_2^2+x_3^2-2,\\
    h_2&=x_1^2+x_2^2-x_3.
\end{align*}

 The real variety
$V_{\R}(I)$ for this ideal is strictly contained in $V_{\C}(I)$. We
set  $\tau=10^{-8}$ and $x_1\tdeg x_2\tdeg x_3$. The rank and corank
sequences for truncated moment matrices $M_{t-\ell}(y)$ are  shown
in Table 7 and 8.

\begin{table*}[!htb]
\begin{minipage}{0.47\textwidth}
\begin{center}
\caption{\small The rank of $M_{t-\ell}(y)$ }
\begin{tabular}{cccccccc}
  \hline\noalign{\smallskip}
Order & $\ell=0$ & $\ell=1$ &$ \ell=2$ \\
\noalign{\smallskip}\hline\noalign{\smallskip}
   t=3  &  7&  5& 3 \\
   t=4  &  9&  7& 5\\
   t=5  & 11&  9& 7\\
   t=6  & 13& 11& 9\\
   \noalign{\smallskip}\hline
\end{tabular}
\end{center}
\end{minipage}
\begin{minipage}{0.50\textwidth}
\begin{center}
\caption{\small The corank of $M_{t-\ell}(y)$ }
\begin{tabular}{cccccccc}
  \hline\noalign{\smallskip}
Order & $\ell=0$ & $\ell=1$ &$ \ell=2$ \\
\noalign{\smallskip}\hline\noalign{\smallskip}
   t=3  & 13&  5& 1 \\
   t=4  & 26& 13& 5 \\
   t=5  & 45& 26& 13\\
   t=6  & 71& 45& 26\\
   \noalign{\smallskip}\hline
\end{tabular}
\end{center}
\end{minipage}
\end{table*}

\begin{table*}[!htb]
\center \caption{\small The  $\alpha_j$ of a reduced basis of $\ker
M_{t-2}(y)$}
\begin{tabular}{ccccccccc}
   \hline\noalign{\smallskip}
Order & $\alpha_1$ &$ \alpha_2$& $\alpha_3$ & $\sum_{j=1}^3 j\alpha_{j}$ \\
\noalign{\smallskip}\hline\noalign{\smallskip}
   t=3  & 0 & 0 & 1 & 0$\times$1+0$\times$2+1$\times$3=3\\
   t=4  & 1 & 2 & 1 & 1$\times$1+2$\times$2+1$\times$3=8\\
   t=5  & 4 & 3 & 1 & 4$\times$1+3$\times$2+1$\times$3=13\\
   t=6  & 8 & 4 & 1 & 8$\times$1+4$\times$2+1$\times$3=19\\
   \noalign{\smallskip}\hline
\end{tabular}

\end{table*}

For t=4, we have
\[\sum_{j=1}^3 j\alpha_{j}=8, ~{\text{and}}~
\crk \, M_{4-1}-\crk \, M_{4-2}=8.\] Hence,   the condition
(\ref{e,5}) is satisfied. The Pommaret basis  computed by
Algorithm \ref{e,6} for $t=4$ is
\[
\{-1+x_3, -1+x_1^2+x_2^2\}. 
\]
 For this example, we can show that
$\langle\ker M_{4-2}(y)\rangle=\sqrt[\R]{I}$, and a reduced   basis of $\ker
M_{4-2}(y)$  is a Pommaret basis of $\sqrt[\R]{I}$. Hence,
 the condition (\ref{e,5}) can
be satisfied by arbitrary  $t \geq 4$.
\end{example}

\begin{example}
Consider the ideal $I=\langle h_1,h_2,h_3\rangle$ in \cite[p.61, Ex
2.4.12]{Seiler02}, where
\begin{align*}
    h_1 &= x_3^2+x_2x_3-x_1^2,\\
    h_2 &= x_1x_3+x_1x_2-x_3,\\
    h_3 &= x_2x_3+x_2^2+x_1^2-x_1.
\end{align*}

Let  $\tau=10^{-7}$ and the term order be   $x_1\tdeg x_2\tdeg x_3$.
The rank and corank sequences for the truncated moment matrices
$M_{t-\ell}(y)$ are shown in  Table 10 and 11. 
\begin{table*}[!htb]
\begin{minipage}{0.48\textwidth}
\begin{center}
\caption{\small The rank of $M_{t-\ell}(y)$ }
\begin{tabular}{cccccccc}
  \hline\noalign{\smallskip}
Order & $\ell=0$ & $\ell=1$ &$ \ell=2$ \\
\noalign{\smallskip}\hline\noalign{\smallskip}
   t=4  & 13& 10& 7 \\
   t=5  & 16& 13& 10\\
   t=6  & 19& 16& 13\\
   t=7  & 22& 19& 16\\
   \noalign{\smallskip}\hline
\end{tabular}
\end{center}
\end{minipage}
\begin{minipage}{0.5\textwidth}
\begin{center}
\caption{\small The corank of $M_{t-\ell}(y)$ }
\begin{tabular}{cccccccc}
  \hline\noalign{\smallskip}
Order & $\ell=0$ & $\ell=1$ &$ \ell=2$ \\
\noalign{\smallskip}\hline\noalign{\smallskip}
   t=4  & 22& 10& 3 \\
   t=5  & 40& 22& 10 \\
   t=6  & 65& 40& 22\\
   t=7  & 98& 65& 40\\
   \noalign{\smallskip}\hline
\end{tabular}
\end{center}
\end{minipage}
\end{table*}\\

\begin{table*}[!htb]
\center \caption{\small The $\alpha_j$ of a reduced basis of $\ker
M_{t-2}(y)$}
\begin{tabular}{ccccccccc}
   \hline\noalign{\smallskip}
Order & $\alpha_1$ &$ \alpha_2$& $\alpha_3$ &$\sum_{j=1}^3 j\alpha_{j}$ \\
\noalign{\smallskip}\hline\noalign{\smallskip}
   t=4  & 1 & 1 & 1 & 1$\times$1+1$\times$2+1$\times$3=6\\
   t=5  & 4 & 2 & 1 & 4$\times$1+2$\times$2+1$\times$3=11\\
   t=6  & 8 & 3 & 1 & 8$\times$1+3$\times$2+1$\times$3=17\\
   t=7  & 13& 4 & 1 & 13$\times$1+4$\times$2+1$\times$3=24\\
   \noalign{\smallskip}\hline
\end{tabular}
\end{table*}
 The condition (\ref{e,5}) can not be
 satisfied for $t$ from $4$ to $7$.  Actually, Seiler showed in
 \cite{Seiler02} that
the coordinates $(x_1, x_2, x_3)$ are not  $\delta$-regular for the
ideal $I$. However, if we perform the linear transformation
suggested  in \cite{Seiler02}, $\tilde{x}_1=x_3$,
$\tilde{x}_2=x_2+x_3$, $\tilde{x}_3=x_1$,  after an auto-reduction,
we obtain the polynomial system
$\tilde{P}=\{\tilde{x}_1\tilde{x}_2-\tilde{x}_3^2,
\tilde{x}_2\tilde{x}_3-\tilde{x}_1, \tilde{x}_2^2-\tilde{x}_3\}$. We
choose an ordering $\tilde{x}_1\tdeg \tilde{x}_2\tdeg \tilde{x}_3$
and $\tau=10^{-8}$. The rank and corank sequences for the truncated
moment matrices $M_{t-\ell}(y)$ are shown in Table 13 and 14.

\begin{table*}[!htb]
\begin{minipage}{0.48\textwidth}
\begin{center}
\caption{\small The rank of $M_{t-\ell}(y)$ }
\begin{tabular}{cccccccc}
  \hline\noalign{\smallskip}
Order & $\ell=0$ & $\ell=1$ &$ \ell=2$ \\
\noalign{\smallskip}\hline\noalign{\smallskip}
   t=4  & 13& 10& 7 \\
   t=5  & 16& 13& 10 \\
   t=6  & 19& 16& 13\\
   \noalign{\smallskip}\hline
\end{tabular}
\end{center}
\end{minipage}
\begin{minipage}{0.5\textwidth}
\begin{center}
\caption{\small The corank of $M_{t-\ell}(y)$ }
\begin{tabular}{cccccccc}
  \hline\noalign{\smallskip}
Order & $\ell=0$ & $\ell=1$ &$ \ell=2$ \\
\noalign{\smallskip}\hline\noalign{\smallskip}
   t=4  & 22& 10& 3  \\
   t=5  & 40& 22& 10 \\
   t=6  & 65& 40& 22 \\
   \noalign{\smallskip}\hline
\end{tabular}
\end{center}
\end{minipage}
\end{table*}

\begin{table*}[!htb]
\center \caption{\small The  $\alpha_j$ of a reduced basis of $\ker
M_{t-2}(y)$}
\begin{tabular}{ccccccccc}
   \hline\noalign{\smallskip}
Order & $\alpha_1$ &$ \alpha_2$& $\alpha_3$ &$\sum_{j=1}^3 j\alpha_{j}$ \\
\noalign{\smallskip}\hline\noalign{\smallskip}
   t=4  & 0 & 2 & 1 & 0$\times$1+2$\times$2+1$\times$3=7\\
   t=5  & 3 & 3 & 1 & 3$\times$1+3$\times$2+1$\times$3=12\\
   t=6  & 7 & 4 & 1 & 7$\times$1+4$\times$2+1$\times$3=18\\
   \noalign{\smallskip}\hline
\end{tabular}

\end{table*}

For t=4, we have
\[\sum_{j=1}^3 j\alpha_{j}=7, ~{\text{and}}~
\crk \, M_{4-1}-\crk \, M_{4-2}=7.\] Hence,   the condition
(\ref{e,5}) is satisfied. The Pommaret basis  computed by
Algorithm \ref{e,6} for $t=4$ is
\[
\{-x_3+x_2^2, -x_1+x_2 x_3, -x_1 x_2 +x_3^2\}. 
\]

\end{example}

\begin{example}
Consider the ideal $I=\langle h_1, h_2\rangle$, where
\begin{align*}
    h_1 
      &= (x_1-x_2) (x_1+x_2)^2(x_1+x_2^2+x_2),\\
    h_2 
      &=(x_1-x_2)(x_1+x_2)^2 (x_1^2+x_2^2).
\end{align*}
 In this example, $I$ is not a radical ideal.  We
set $\tau=10^{-4}$ and $x_1\tdeg x_2$. The rank and corank sequences
for the truncated moment matrices $M_{t-\ell}(y)$ are shown in Table 16 and 17.
\begin{table*}[!htb]
\begin{minipage}{0.48\textwidth}
\begin{center}
\caption{\small The rank of $M_{t-\ell}(y)$ }
\begin{tabular}{cccccccc}
  \hline\noalign{\smallskip}
Order & $\ell=0$ & $\ell=1$ &$ \ell=2$ \\
\noalign{\smallskip}\hline\noalign{\smallskip}
   t=7 & 15& 13& 11 \\
   t=8 & 17& 15& 13\\
   t=9 & 19& 17& 15\\
   \noalign{\smallskip}\hline
\end{tabular}
\end{center}
\end{minipage}
\begin{minipage}{0.50\textwidth}
\begin{center}
\caption{\small The corank of $M_{t-\ell}(y)$ }
\begin{tabular}{cccccccc}
  \hline\noalign{\smallskip}
Order & $\ell=0$ & $\ell=1$ &$ \ell=2$ \\
\noalign{\smallskip}\hline\noalign{\smallskip}
   t=7  & 21& 15& 10\\
   t=8  & 28& 21& 15 \\
   t=9  & 36& 28& 21\\
   \noalign{\smallskip}\hline
\end{tabular}
\end{center}
\end{minipage}
\end{table*}

\begin{table*}[!htb]
\center \caption{\small The  $\alpha_j$ of a reduced basis of $\ker
M_{t-2}(y)$}
\begin{tabular}{ccccccccc}
   \hline\noalign{\smallskip}
Order & $\alpha_1$ &$ \alpha_2$&$\sum_{j=1}^2 j\alpha_{j}$ \\
\noalign{\smallskip}\hline\noalign{\smallskip}
   t=7  & 3 & 1 & 3$\times$1+1$\times$2=5\\
   t=8  & 4 & 1 & 4$\times$1+1$\times$2=6\\
   t=9  & 5 & 1 & 5$\times$1+1$\times$2=7\\
   \noalign{\smallskip}\hline
\end{tabular}

\end{table*}

For t=7, we have
\[\sum_{j=1}^2 j\alpha_{j}=5, ~{\text{and}}~
\crk \, M_{7-1}-\crk \, M_{7-2}=5.\] Hence,   the condition
(\ref{e,5}) is satisfied. The Pommaret basis  computed by
Algorithm \ref{e,6} for $t=7$ is
\[
\{-x_1^2+x_2^2\}. 
\]
It should be noticed that for this example, if we set tolerance
$\tau < 10^{-4}$, the rank and corank sequences for the truncated
moment matrices $M_{t-\ell}(y)$ will be completely different from
those shown in Table 16 and 17, and we can not get
\{$-x_1^2+x_2^2\}$ as a 
Pommaret basis of $\sqrt[\R]{I}$.
\end{example}

\begin{example}
We  compute  $I(V_{\R}(I)\cap \A)$ for  $I=\langle h_1, h_2\rangle$,
\begin{align*}
    h_1 
        &=(x_1-x_2) (x_1+x_2)(x_1+x_2^2+x_2),\\  
    h_2 
     &=(x_1-x_2)(x_1+x_2)(x_1^2+x_2^2),
\end{align*}
and 
\[ \A=\{(x_1,x_2) \in \R^2 ~|~ x_1\geq 1, ~ x_2\geq 1\}.\]

Let us set $\tau=10^{-8}$ and $x_1\tdeg x_2$, the rank and corank
sequences for the truncated moment matrices $M_{t-\ell}(y)$ with
$y\in\Kk_{t,\A}$ are shown in Table 19 and 20.
\begin{table*}[!htb]
\begin{minipage}{0.48\textwidth}
\begin{center}
\caption{\small The rank of $M_{t-\ell}(y)$ }
\begin{tabular}{cccccccc}
  \hline\noalign{\smallskip}
Order & $\ell=0$ & $\ell=1$ &$ \ell=2$ \\
\noalign{\smallskip}\hline\noalign{\smallskip}
   t=6 &  8& 6&5\\
   t=7 &  9& 7&6\\
   t=8 & 10& 8&7\\
   \noalign{\smallskip}\hline
\end{tabular}
\end{center}
\end{minipage}
\begin{minipage}{0.50\textwidth}
\begin{center}
\caption{\small The corank of $M_{t-\ell}(y)$ }
\begin{tabular}{cccccccc}
  \hline\noalign{\smallskip}
Order & $\ell=0$ & $\ell=1$ &$ \ell=2$ \\
\noalign{\smallskip}\hline\noalign{\smallskip}
   t=6  & 20& 15& 10\\
   t=7  & 27& 21& 15 \\
   t=8  & 35& 28& 21\\
   \noalign{\smallskip}\hline
\end{tabular}
\end{center}
\end{minipage}
\end{table*}

\begin{table*}[!htb]
\center \caption{\small The  $\alpha_j$ of a reduced basis of $\ker
M_{t-2}(y)$}
\begin{tabular}{ccccccccc}
   \hline\noalign{\smallskip}
Order & $\alpha_1$ &$ \alpha_2$&$\sum_{j=1}^2 j\alpha_{j}$ \\
\noalign{\smallskip}\hline\noalign{\smallskip}
   t=6  & 3 & 1 & 3$\times$1+1$\times$2=5\\
   t=7  & 4 & 1 & 4$\times$1+1$\times$2=6\\
   t=8  & 5 & 1 & 5$\times$1+1$\times$2=7\\
   \noalign{\smallskip}\hline
\end{tabular}
\end{table*}

For t=6, we have
\[\sum_{j=1}^2 j\alpha_{j}=5, ~{\text{and}}~
\crk \, M_{6-1}-\crk \, M_{6-2}=5.\] Hence,   the condition
(\ref{e,5}) is satisfied. The Pommaret basis we obtain  by Algorithm
\ref{e,6} for $t=6$ is %
%
%
\[
\{-x_1+x_2\}
\]
 for  $I(V_{\R}(I)\cap \A)$.
\end{example}

\section{Conclusion}\label{conclusions}

In this paper we present a semidefinite  characterization for
computing  a  Pommaret basis of an ideal $J$, where $J$ is generated
by polynomials in the kernel of a truncated moment matrix and
satisfies $I \subseteq J \subseteq I(V_{\R}(I))$. Our approach is
stimulated by the previous work in
\cite{LLR08,LLR09a,LR10,RZ09,Rostalski09,Robin:2006,SRWZ09,Seiler02}.
By combining the geometric involutive theory with the results on
positive semidefinite moment matrices, we introduce a new stopping
condition (\ref{e,5}) for the semidefinite program (\ref{e,11}) and
prove the finite termination of the algorithm in a $\delta$-regular
coordinate system. Although  from the tables in Section
\ref{examples}, we can check that the condition  (\ref{e,5}) can be
satisfied by higher order moment matrices once it is satisfied at
some order, in general, we can not guarantee this property.
Therefore, unlike flat extension conditions proposed by Curto and
Fialkow in \cite{CF96} for finite rank moment matrices, we can not
prove the computed Pommaret basis is an involutive basis of the real
radical ideal. Finally, we wish to mention that results computed by
semidefinite programming and numerical linear algebra are
approximate. Therefore, our condition (\ref{e,5}) can only be
checked with respect to a given tolerance. For improperly chosen
tolerance, we might not be able to give a meaningful 
 answer.

\bigskip
\noindent {\bf Acknowledgement} \,  We are most grateful to Jiawang
Nie for  many constructive remarks to improve the presentation of
the paper. Yue Ma, Chu
 Wang and Lihong Zhi were partially supported
by  a NKBRPC 2011CB302400, and the Chinese National Natural Science
Foundation under grant NSFC 91118001, 60911130369 and 60821002/F02.

\bibliographystyle{plain}

\bibliography{polysolve}
\end{document}